\documentclass[11pt]{article}
\usepackage{makeidx}

\usepackage{amsmath,amssymb}
\usepackage{amsthm}
\usepackage{graphicx}
\usepackage{color}
\usepackage{verbatim}
\usepackage{cases}
\newcommand{\mb}[1]{\mbox{\boldmath $#1$}}

\definecolor{blue}{rgb}{0,0,0.9}
\definecolor{red}{rgb}{0.9,0,0}
\definecolor{green}{rgb}{0,0.9,0}

\newcommand{\cX}{{\cal X}}
\newcommand{\cT}{{\cal T}}

\newcommand{\cU}{{\cal U}}
\newcommand{\cH}{{\cal H}}
\newcommand{\cM}{{\cal M}}

\newcommand{\cP}{{\cal P}}

\newcommand{\cA}{{\cal A}}

\newcommand{\cQ}{{\cal Q}}
\newcommand{\cB}{{\cal B}}

\newcommand{\bea}{\begin{eqnarray*}}
\newcommand{\eea}{\end{eqnarray*}}

\newcommand{\inprod}[2]{\langle #1 , #2 \rangle }

\newcommand{\bc}{\begin{center}}
\newcommand{\ec}{\end{center}}

\newcommand{\R}{\mathbb R}

\newcommand{\be}{\begin{equation}}
\newcommand{\ee}{\end{equation}}
\newcommand{\beaa}{\begin{eqnarray*}}
\newcommand{\eeaa}{\end{eqnarray*}}
\newcommand{\ben}{\begin{enumerate}}
\newcommand{\een}{\end{enumerate}}
\newcommand{\db}{\hspace*{\fill}{\zapf o}}

\newcommand{\bpn}{\begin{proposition}\twlsf}
\newcommand{\epn}{\db\end{proposition}}
\newcommand{\bdm}{\begin{displaymath}}
\newcommand{\edm}{\end{displaymath}}
\newcommand{\ba}{\begin{array}}
\newcommand{\ea}{\end{array}}

\newcommand{\argmin}{\mathop{\rm argmin}}

\def\texitem#1{\par\smallskip\noindent\hangindent 25pt
               \hbox to 25pt {\hss #1 ~}\ignorespaces}
\newcommand{\norm}[1]{\left\lVert#1\right\rVert}

\newtheorem{assumption}{Assumption}

\newtheorem{lemma}{Lemma}
\newtheorem{proposition}{Proposition}

\newtheorem{remark}{Remark}
\newtheorem{theorem}{Theorem}

\renewcommand{\S}{\mathbb{S}}

\def\lam{\lambda} 
\def\sig{\sigma}
\def\eps{\epsilon}

\def\nn{\nonumber}
\def\inprod#1#2{\langle#1,\,#2\rangle}
\def\norm#1{\|#1\|}
\def\diag#1{\mbox{diag}(#1)}

\def\gam{\gamma}

 \def\cV{{\cal V}}
 \def\cI{{\cal I}}  \def\cD{{\cal D}}
\def\cB{{\cal B}} 
 \def\cW{{\cal W}}

\def\mb0{\mbox{\bf 0}}
\def\ome{\omega}
\def\x{\mbox{\boldmath{$x$}}} 
\def\b{\mbox{\boldmath{$b$}}}
\def\g{\mbox{\boldmath{$g$}}}
\def\e{\mbox{\boldmath{$e$}}}
\def\bdelta{\mbox{\boldmath{$\delta$}}}
\def\bgamma{\mbox{\boldmath{$\gamma$}}}
\def\cO{{\Theta}}
\def\cY{{\cal Y}}
\def\hcQ{{\widehat{\cQ}}}

\begin{document}

\title{A block symmetric Gauss-Seidel decomposition theorem for convex composite
quadratic programming and  its applications}

\author{
 Xudong Li\thanks{Department of Mathematics, National
         University of Singapore,
        10 Lower Kent Ridge Road, Singapore
         119076 ({\tt matlixu@nus.edu.sg}).}, \;
 Defeng Sun\thanks{Department of Mathematics and Risk Management Institute, National
         University of Singapore,
        10 Lower Kent Ridge Road, Singapore
         119076 ({\tt matsundf@nus.edu.sg}).} , \;
 Kim-Chuan Toh\thanks{Department of Mathematics, National
         University of Singapore,
        10 Lower Kent Ridge Road, Singapore
         119076 ({\tt mattohkc@nus.edu.sg}).}
}

\date{May 16, 2017}
\maketitle

\begin{abstract}
For a symmetric positive semidefinite linear system of equations $\cQ \x = \b$, where $\x = (x_1,\ldots,x_s)$ is
partitioned into $s$ blocks, with $s \geq 2$,
we show that
each cycle of the classical block symmetric Gauss-Seidel (block sGS) method exactly solves the associated quadratic programming (QP)
problem but  added with an extra proximal term of the form $\frac{1}{2}\norm{\x-\x^k}_\cT^2$,
 where $\cT$ is a symmetric positive semidefinite matrix related to the sGS decomposition of $\cQ$
 and  $\x^k$ is the previous iterate.
By leveraging on such a connection to optimization, we are able to extend the result (which we name as the  block sGS decomposition theorem)
for solving a convex composite
QP (CCQP) with an additional possibly nonsmooth term in $x_1$, i.e., $\min\{ p(x_1) + \frac{1}{2}\inprod{\x}{\cQ\x} -\inprod{\b}{\x}\}$, where $p(\cdot)$ is a proper closed convex function.
Based on the block sGS decomposition theorem, we 
extend the classical block sGS method to solve a CCQP.
In addition, our extended block sGS method has the flexibility of allowing for inexact computation in each
step of the block sGS cycle. At the same time, we can also accelerate the inexact block sGS method to achieve an iteration complexity of
$O(1/k^2)$ after performing $k$ 
cycles. As a {fundamental} building block,
the block sGS decomposition theorem has played a key role in various recently developed algorithms
such as the inexact semiproximal {ALM/ADMM} for linearly constrained multi-block convex composite conic programming (CCCP), and
the accelerated block coordinate descent method for multi-block CCCP.
\end{abstract}


\section{Introduction}

It is well known that the classical block symmetric Gauss-Seidel (block sGS) method \cite{Axelsson,Greenbaum,Hackbusch, Saad}
can be used to solve a
symmetric positive semidefinite linear system of equations $\cQ \x = \b$ where $\x = (x_1;\ldots;x_s)$ is partitioned
into $s$ blocks with $s\geq 2$. We are particularly interested in  the case when $s > 2$.
In this paper, we show that
each cycle of the classical block sGS method exactly solves the corresponding convex quadratic programming (QP)
problem but added with an extra proximal term depending on the previous iterate (say $\x^k$).
Through such a connection to optimization, we are able to  extend the result (which we name as the
block sGS decomposition theorem) to a convex composite
QP (CCQP) with an additional possibly nonsmooth term in $x_1$, and subsequently
extend the classical block sGS method to solve a CCQP.
We can also extend the classical block sGS method to the inexact setting, where the underlying linear system
for each block of the new iterate $\x^{k+1}$ need not be solved exactly.
Moreover,  by borrowing ideas
 in the optimization literature, we are able to accelerate the classical block sGS method and
provide new convergence results. More details will be given later.

Assume that $\cX_i = \R^{n_i}$ for $i=1,\ldots,s$, and $\cX = \cX_1\times\cdots\times \cX_s$,
where $s \geq 2$ is a given integer. 
Consider the following symmetric positive semidefinite block linear system of equations:
\begin{eqnarray}
  \cQ \x = \b,
\label{eq-1}
\end{eqnarray}
where  $\x = [x_1; \; \ldots;\; x_s]\in \cX$,
$\b = [b_1; \; \ldots;\; b_s] \in \cX$, and
\begin{eqnarray}
\cQ &=& \left[ \begin{array}{ccc}
 Q_{1,1} &\dots  & Q_{1,s}  \\[5pt]
 \vdots           &\vdots         & \vdots                \\[5pt]
Q_{1,s}^* &\dots  & Q_{s,s}
\end{array}\right]
\label{eq-Q}
\end{eqnarray}
with $Q_{i,j} \in \R^{n_i\times n_j}$ for $1\le i,j\le s$.
It is well known that \eqref{eq-1} is the optimality condition for the following unconstrained QP:
\begin{eqnarray}
\mbox{(QP)} \quad \min \Big\{ q(\x) := \frac{1}{2}\inprod{\x}{\cQ\x} -\inprod{\b}{\x} \mid \x \in \cX \Big\}.
\label{eq-QP0}
\end{eqnarray}
Note that
even though our problem is phrased in the matrix-vector setting for convenience, one
can consider the setting where each $\cX_i$ is a real $n_i$-dimensional inner product space and
$Q_{i,j}$ is a linear map from $\cX_i$ to $\cX_j$.
Throughout the paper, we make the following assumption:
\begin{assumption}
 {$\cQ$ is symmetric positive semidefinite
and each diagonal block $Q_{i,i}$ is symmetric positive definite for $i=1,\ldots,s$.} \\
\end{assumption}
From  the following decomposition of $\cQ$:
\begin{eqnarray}
 \cQ &=& \cU + \cD + \cU^*,
\label{eq-splitting}
\end{eqnarray}
where
\begin{eqnarray}
\label{eq-UD}
  \cU =
 \left[ \begin{array}{cccc}
  \mb0 & Q_{1,2} &\dots & Q_{1,s} \\
           & \ddots   &        & \vdots \\
           &              & \ddots & Q_{s-1,s} \\[5pt]
           &              &            &\mb0
\end{array}\right],
\quad \cD =  \left[ \begin{array}{cccc}
  Q_{1,1} &  & \\[3pt]
           & Q_{2,2}   &        & \\
           &              & \ddots & \\
           &              &            &Q_{s,s}
\end{array}\right],
\end{eqnarray}
 the classical block sGS iteration in numerical analysis is usually derived as a natural generalization of the
classical pointwise sGS for solving a symmetric {\em positive definite} linear system of equations, and the latter is
typically derived as a fixed-point iteration for  the sGS matrix splitting based on \eqref{eq-splitting};
{see for example \cite[Sec. 4.1.1]{Saad}, \cite[Sec. 4.5]{Hackbusch}}. Specifically,
 the block sGS fixed-point iteration in the third normal form (in the terminology used in \cite{Hackbusch}) reads as follows:
\begin{eqnarray}
   \widehat{\cQ} (\x^{k+1} -\x^k) = \b-\cQ \x^k,
\label{eq-SGS}
\end{eqnarray}
where $\widehat{\cQ} = (\cD+\cU) \cD^{-1} (\cD+\cU^*)$.

In this paper, we give a derivation of the classical block sGS method \eqref{eq-SGS}
from the optimization perspective. By doing so, we
are able to extend the classical block sGS method to solve a structured CCQP
problem of the form:
\begin{eqnarray}
 \mbox{(CCQP)} \quad \min \Big \{ F(\x) \;:=\; p(x_1) + \frac{1}{2}\inprod{\x}{\cQ\x} -\inprod{\b}{\x} \mid \x = [x_1;\ldots;x_s] \in \cX \Big\},
\label{eq-QP}
\end{eqnarray}
where $p:\cX_1 \rightarrow (-\infty,\infty]$ is a proper closed convex function such as
$p(x_1) = \norm{x_1}_1$ or $p(x_1) = \delta_{\R^{n_1}_+}(x_1)$ (the indicator function of $\R^{n_1}_+$
defined by $\delta_{\R^{n_1}_+}(x_1) = 0$
if $x_1 \in \R^{n_1}_+$ and $\delta_{\R^{n_1}_+}(x_1) = \infty$ otherwise).
Our specific contributions are described in the next few paragraphs. We note that
the main results presented here are parts of the thesis of the first author \cite{LiThesis2014}.

{ First}, we establish the key result of the paper,
the {\it block sGS decomposition theorem}, which states that
each cycle of the block sGS method, say at the $k$th iteration, corresponds exactly to solving
\eqref{eq-QP} with an additional proximal term $\frac{1}{2}\norm{\x - \x^k}^2_{\cT_\cQ}$ added
to its objective function, i.e.,
\begin{eqnarray}
\min\Big\{p(x_1) +  \frac{1}{2} \inprod{\x}{\cQ \x} -\inprod{\b}{\x} + \frac{1}{2}\norm{\x-\x^k}_{\cT_\cQ}^2\mid \x \in \cX \Big\},
\label{eq-QPsub}
\end{eqnarray}
 where $\cT_\cQ = \cU\cD^{-1}\cU^*$, and $\norm{\x}_{\cT_\cQ}^2 = \inprod{\x}{\cT_\cQ \x}.$
It is clear that when $p(\cdot) \equiv 0$,  the problem \eqref{eq-QP} is exactly
the QP \eqref{eq-QP0} associated with the linear system
\eqref{eq-1}.
Therefore, we can interpret the classical block sGS method as a proximal-point  minimization method
for solving the QP \eqref{eq-QP0},
and each cycle of the classical block sGS method
solves  exactly the proximal subproblem \eqref{eq-QPsub} associated with the QP \eqref{eq-QP0}.
As far as we are aware of, this is the first time in which the classical block sGS method \eqref{eq-SGS} (and also the pointwise sGS method) is derived from an optimization perspective.

{ Second,} we also establish a factorization view of the block sGS decomposition theorem and
show its equivalence to the
Schur complement based (SCB) reduction procedure proposed in \cite{SCBADMM} for solving
a recursively defined variant of the
proximal subproblem \eqref{eq-QPsub}. The SCB reduction procedure in \cite{SCBADMM} is derived
by inductively finding an appropriate proximal term to be added to the objective function of
\eqref{eq-QP} so that the block variables $x_s,x_{s-1},\ldots,x_2$ can be eliminated in a
sequential manner and thus ending with a minimization problem involving only the variable $x_1$.
In a nutshell, we show that the SCB reduction procedure sequentially
eliminates the blocks (in the reversed order starting from $x_s$)
in the variable $\x$ of  the proximal subproblem \eqref{eq-QPsub}
by decomposing the
proximal term $\frac{1}{2}\norm{\x-\x^k}^2_{\cT_\cQ}$ also in a sequential manner.
In turn, each of the reduction step corresponds exactly to one step in a cycle of the block sGS method.

{Third}, based on the block sGS decomposition theorem,
we are able to extend the classical block sGS method for solving the QP \eqref{eq-QP0}
to solve the CCQP \eqref{eq-QP}, and each cycle of the extended block sGS method
corresponds
precisely to solving the proximal subproblem \eqref{eq-QPsub}.
Our extension of the block sGS method has thus overcome the limitation
of the classical method by allowing us to solve the nonsmooth CCQP which
often arises in practice, for example, in semidefinite programming
{where $p(x_1) = \delta_{\S^{n_1}_+}(x_1)$ and $\S^{n_1}_+$ is the cone of
$n_1\times n_1$ symmetric positive semidefinite matrices.
Moreover, our extension also allows the updates of the blocks to be inexact.
As a consequence, we also obtain an inexact version of the classical block sGS method,
where the iterate $\x^{k+1}$ need not be computed exactly from \eqref{eq-SGS}.
We should emphasize that
the inexact block sGS method is potentially very useful when a diagonal block, say $Q_{i,i}$,  in
\eqref{eq-Q} is large and the computation of $\x^{k+1}_i$ must be done via an iterative solver rather
than a direct solver. Note that even for the linear system \eqref{eq-SGS},
our systematic approach (in section 4)  to derive the inexact extension of the classical block sGS method
appears to be new. The only inexact variant of the classical
block sGS method for \eqref{eq-SGS} with a convergence proof we are aware of
is the pioneering work of Bank et al. in  \cite{BDY}. In \cite{BDY}, the authors showed that
by modifying the diagonal blocks in $\cD$,
the linear system involved in each block can be solved by a given fixed number of pointwise sGS cycles.
}

Fourth, armed with the optimization interpretation of each cycle of the block sGS method, it becomes
easy for us to adapt ideas from the optimization literature to establish the
iteration complexity of $O(\norm{\x^0-\x^*}_{\hcQ}^2/k)$ for the
extended block sGS method as well as to accelerate it to obtain the complexity
of $O(\norm{\x^0-\x^*}_{\hcQ}^2/(k+1)^2)$, after running for $k$ cycles,  {where $\x^*$ is an optimal solution for \eqref{eq-QP}}.
Just as in
 the classical block sGS method, we  {can} obtain a linear rate of convergence for our extended inexact block sGS method under the assumption that $\cQ$ is positive definite.
 With the help of an extensive optimization literature on the linear convergences
 of proximal gradient methods, we are further able to relax the positive definiteness assumption on $\cQ$ to
 a mild error bound assumption on the function $F$  in \eqref{eq-QP}
 and derive at least R-linear convergence results for our extended block sGS method.
 The error bound assumption in fact
 holds automatically for many interesting applications, including the important case when
 $p(\cdot)$ is a piecewise linear-quadratic function.
 We note that there is active research in studying the convergence of proximal
 gradient methods for a convex composite minimization problem of the form $\min\{ f(\x) + g(\x) \mid \x \in \cX\}$, with
 $f$ being a smooth convex function and $g$ a proper closed convex function whose proximal map is easy to compute;
 see for example
 \cite{SRB} and the references therein.
 In each iteration of a typical proximal gradient method,
 a simple proximal term $\frac{L}{2}\norm{\x-\bar{\x}}^2$, where
 $L$ is a Lipschitz constant for the gradient of $f$,
 is added to the objective function.
 Our extended block sGS method for (CCQP) differs
 from those proximal gradient methods in the literature in that the proximal term we add
 comes from the
 sophisticated positive semidefinite  linear operator    associated with the sGS decomposition of $\cQ.$

Recent research works in \cite{CST17,QSDPNAL,SCBADMM,STY16,STY15} have shown that  our block sGS decomposition theorem for the
CCQP \eqref{eq-QP} can play an essential role
in the design of efficient algorithms for solving various convex  optimization problems such as
convex composite quadratic semidefinite programming problems.
{Indeed, the block sGS decomposition based ADMM algorithms designed in \cite{CST17,STY16,STY15}
have found applications in various recent papers such as \cite{BZNC,FKS,DWD}.}
Our experiences have shown that the inexact block sGS cycle can provide the much needed 
computational efficiency
when
one is
designing an algorithm based on the framework of the proximal augmented Lagrangian (ALM) or proximal
alternating direction method of multipliers (ADMM)  for solving important classes of large scale convex composite optimization problems.
As a concrete illustration of the application of our
block sGS decomposition theorem, we will briefly describe in section \ref{sec-pALM} on
how to utilize the theorem in
the
design of the  proximal augmented Lagrangian method for solving
a linearly constrained convex composite quadratic programming problem.

The idea of sequentially updating the blocks of a multi-block variable, either in the
Gauss-Seidel fashion or the successive over-relaxation (SOR) fashion, has been
incorporated into quite a number of optimization algorithms \cite{Bertsekas}
and in solving nonlinear equations \cite{Ortega}.
Indeed the Gauss-Seidel  (also known as the block coordinate descent) approach for solving optimization problems has been considered extensively; we refer the
readers to
\cite{Beck, Grippo} for the literature review on the recent developments, especially for the case where $s>2$.
Here we would like to emphasize that even for the case of an unconstrained {\em smooth} convex
minimization problem $\min\{f(\x)\mid \x\in\cX\}$, whose objective function $f(\x)$ (not necessarily strongly convex) has a Lipschitz continuous gradient of modulus $L$, it is only proven recently in \cite{Beck} that the block coordinate (gradient) descent method
is globally convergent with the iteration complexity of $O(L s/k)$ after $k$ cycles, where $s$ is
the number of blocks. When $f(\x)$ is the quadratic function in \eqref{eq-QP0}, the block coordinate descent
method is precisely  the classical block Gauss-Seidel (GS) method. In contrast to the block sGS method,
each iteration of the block GS method does not appear to have  an optimization equivalence.
Despite the extensive work on the Gauss-Seidel approach for solving convex optimization problems,
surprisingly, little is known about the
symmetric Gauss-Seidel approach
for solving the same problems except for the recent paper \cite{STY16} which
utilized our block sGS decomposition theorem to design an inexact accelerated block
coordinate descent method to solve a problem of the form
$\min \{ p(x_1) + f(\x) \mid \x\in \cX\}.$

The remaining parts of the paper are organized as follows. The next section is devoted to
the block sGS decomposition theorem for the CCQP \eqref{eq-QP}.
In section \ref{SCBandsGS}, we present a factorization
view of the block sGS theorem and prove its equivalence to the SCB reduction procedure
proposed in \cite{SCBADMM,LiThesis2014}.
In the  {following} section,
we derive the block sGS method from an optimization perspective and extend it to solve the
CCQP \eqref{eq-QP}.  The convergence results for our extended
block sGS method are also presented in this section.
In  section \ref{sec-pALM},
the application of our block sGS decomposition theorem is demonstrated in the design of a
proximal augmented Lagrangian method for solving
a linearly constrained convex composite quadratic  programming problem.
The extension of the classical block symmetric SOR method
for solving \eqref{eq-QP} is presented in section \ref{sSOR}. We conclude our paper in the final section.

We end the section by giving some notation. For a symmetric matrix
$\cQ$,  the notation $\cQ \succeq 0$ ($\cQ \succ 0$) means that
the matrix $\cQ$ is symmetric positive semidefinite (definite).
The spectral norm of $\cQ$ is denote by $\norm{\cQ}_2.$

\section{Derivation of the block sGS decomposition theorem for \eqref{eq-QP}   }
\label{sec:sGS}

In this section, we present the derivation of one cycle of the block sGS method for \eqref{eq-QP} from the optimization perspective as mentioned in the introduction.

Recall the decomposition of $\cQ$ in \eqref{eq-splitting}, $\cU$,$\cD$ in \eqref{eq-UD} and the {\it sGS linear operator}
defined by
\begin{equation}\label{eq-Tsgs}
  \cT_{\cQ}  = \cU \cD^{-1} \cU^*.
\end{equation}
Given  $\bar \x\in \cX$, corresponding to problem \eqref{eq-QP}, we consider solving the following subproblem
\begin{eqnarray}
   \x^+ := \mbox{argmin}_{\x\in\cX}\; \Big\{p(x_1)+q(\x)
  +\frac{1}{2}\norm{\x  - \bar \x }_{\cT_\cQ}^2
  -\inprod{\x}{\Delta(\mbox{\boldmath{$\delta'$}},\mbox{\boldmath{$\delta$}})} \Big\},
 \label{prox-T}
\end{eqnarray}
where $\mbox{\boldmath{$\delta'$}},\mbox{\boldmath{$\delta$}}\in\cX$ are two given error
 vectors with $\delta_1' = \delta_1$, and
\begin{eqnarray}\label{eq-Delta}
\Delta(\mbox{\boldmath{$\delta'$}},\mbox{\boldmath{$\delta$}}):= \mbox{\boldmath{$\delta$}} + \cU \cD^{-1}(\mbox{\boldmath{$\delta$}} -\mbox{\boldmath{$\delta'$}}).
\end{eqnarray}
We note that the vectors $\mbox{\boldmath{$\delta'$}},\mbox{\boldmath{$\delta$}}$
need not be known a priori. We should view $\x^+$ as an approximate solution to \eqref{prox-T} without the
perturbation term $\inprod{\x}{\Delta(\mbox{\boldmath{$\delta'$}},\mbox{\boldmath{$\delta$}})}$. Once $\x^+$ has been computed,
the associated error vectors can then be obtained, and $\x^+$ is then the exact solution to the perturbed problem \eqref{prox-T}.

The following theorem
shows that $\x^+$ can be computed by performing exactly one cycle of
the block sGS method for \eqref{eq-QP}.
In particular, if $p(x_1)\equiv 0$ and $\bdelta'=0=\bdelta$, then the computation of $\x^+$ corresponds
exactly   to
one cycle of the classical block sGS method.
For the proof, we need to define the following notation
for a given $\x= (x_1;\ldots; x_s)$,
\begin{eqnarray*}
 x _{\geq i} = (x_i;\ldots; x_s), \quad x_{\leq i} = (x_1;\ldots;x_i), \quad i = 1,\ldots,s.
\end{eqnarray*}
We also define $x_{\geq s+1} = \emptyset$.

\begin{theorem}[sGS Decomposition]\label{thm:nbsGS}
 Assume that $\cQ\succeq 0$ and the self-adjoint linear operators $\cQ_{ii}$ are  positive definite for all $i=1,
\ldots,s$. Then, it holds that
\begin{equation} \label{psdeq}
\widehat{\cQ} :=  \cQ + \cT_{\cQ} = (\cD + \cU)\cD^{-1}(\cD + \cU^*) \succ 0.
\end{equation}
For $i=s, \ldots,2,$ suppose that we have computed $x_i^\prime\in \cX_i$ defined by
\begin{equation}
\begin{aligned}
   x^\prime_i :={}& \argmin_{x_i\in \cX_i} \; p(\bar x_1) + q(\bar x_{\le i-1};x_i;x^\prime_{\ge i+1})-\inprod{\delta_i'}{x_i}
\\[5pt]
     ={}& Q_{ii}^{-1}\big(b_i +\delta_i' - \mbox{$\sum_{j=1}^{i-1}$} Q_{ji}^*\bar x_j - \mbox{$\sum_{j=i+1}^s$} Q_{ij}x^\prime_j \big).
\end{aligned}
  \label{xpimrei}
\end{equation}
Then the optimal solution $\x^+$ for \eqref{prox-T} can be computed exactly via the following steps:
\begin{equation} \label{prox-nT}
\left\{
\begin{aligned}
  x_1^+ ={}& \argmin_{x_1\in\cX_1}\; p(x_1) + q(x_1;x^\prime_{\ge 2}) - \inprod{\delta_1}{x_1},
\\[5pt]
 x_i^+ ={}& \argmin_{x_i\in\cX_i} \; p(x_1^+) + q(x^+_{\le i-1};x_i;x^\prime_{\ge i+1}) - \inprod{\delta_i}{x_i}\\[5pt]
 ={}& Q_{ii}^{-1}\big(b_i+\delta_i -\mbox{$\sum_{j=1}^{i-1}$}
Q_{ji}^*x_j^+ - \mbox{$\sum_{j=i+1}^s$} Q_{ij}x^\prime_j\big),\quad i=2,\ldots,s.
 \end{aligned}
\right.
\end{equation}
\end{theorem}
\begin{proof}
Since $\cD\succ0$, we know that $\cD$, $\cD +\cU$ and $\cD +\cU^*$ are all nonsingular. Then, \eqref{psdeq} can easily be obtained from the following observation
\begin{equation} \label{decomp-HT}
\cQ + \cT_{\cQ} = \cD + \cU + \cU^* + \cU\cD^{-1}\cU^* = (\cD + \cU)\cD^{-1}(\cD + \cU^*).
\end{equation}

Next we show the equivalence between (\ref{prox-T}) and  (\ref{prox-nT}).
By noting that $\delta_1 = \delta_1^\prime$ and $\cQ_{11}\succ 0$,  we can define $x_1^\prime$ as follows:
\begin{eqnarray}
  \label{x1p}
  x_1^\prime = \argmin_{x_1\in \cX_1}\; p(x_1) + q(x_1; x^\prime_{\ge 2}) - \inprod{\delta_1^\prime}{x_1}
  = \argmin_{x_1\in \cX_1}\; p(x_1) + q(x_1; x^\prime_{\ge 2}) - \inprod{\delta_1}{x_1} = x_1^+.
\end{eqnarray}
 The
optimality conditions corresponding to $x_1^\prime$ and $x_1^+$ in \eqref{x1p} can be
written as
\begin{subnumcases}{}
  Q_{11}x_1^\prime = b_1 -\gamma_1 +\delta_1^\prime - \mbox{$\sum_{j=2}^s$} Q_{ij}x^\prime_j,
  \label{optx1p}
\\[3pt]
  Q_{11}x_1^+ = b_1 -\gamma_1 +\delta_1 - \mbox{$\sum_{j=2}^s$} Q_{ij}x^\prime_j,
  \label{optx1+}
\end{subnumcases}
where $\gamma_1\in\partial p(x_1^\prime) \equiv\partial p(x_1^+) $. Simple calculations show that \eqref{optx1p} together with \eqref{xpimrei} can equivalently be
rewritten as
\begin{equation*}
  \label{eq-gs-back}
  (\cD + \cU)\x^\prime = \b - \bgamma + \bdelta^\prime - \cU^*\bar \x,
\end{equation*}
where $\bgamma = (\gamma_1;0; \ldots,0)\in\cX$,
while \eqref{prox-nT} can equivalently be recast as
\begin{equation*}
  \label{eq-gs-foward}
  (\cD + \cU^*)\x^+ = \b - \bgamma + \bdelta - \cU \x^\prime.
\end{equation*}
By substituting $\x^\prime = (\cD + \cU)^{-1}(\b - \bgamma + \bdelta^\prime - \cU^*\bar\x)$ into the above equation, we obtain that
\begin{align*}
  (\cD + \cU^*)\x^+ ={}& \b - \bgamma + \bdelta - \cU(\cD + \cU)^{-1}(\b - \bgamma + \bdelta^\prime - \cU^*\bar \x) \nn\\[5pt]
  ={}& \cD(\cD + \cU)^{-1}(\b - \bgamma) + \cU(\cD + \cU)^{-1}\cU^*\bar \x + \bdelta - \cU(\cD + \cU)^{-1}\bdelta^\prime,  \nn
\end{align*}
which, together with \eqref{decomp-HT}, \eqref{eq-Delta} and the definition of $\cT_{\cQ}$ in \eqref{eq-Tsgs}, implies that
\begin{equation}\label{eq-op-sgs}
(\cQ + \cT_{\cQ}) \x^+ = \b - \bgamma + \cT_\cQ\bar \x + \Delta(\bdelta^\prime,\bdelta).
\end{equation}
In the above, we have used the fact that $(\cD+\cU)\cD^{-1}\cU(\cD+\cU)^{-1}= \cU \cD^{-1}$.
By noting that \eqref{eq-op-sgs} is in fact the optimality condition for
\eqref{prox-T} and $\cQ +\cT_{\cQ}\succ 0$, we have thus obtained the equivalence between (\ref{prox-T}) and
(\ref{prox-nT}).
This completes the proof of the theorem.
\end{proof}

 We shall explain here the roles of the error vectors
 $ \bdelta'$ and $ \bdelta$ in the above block  sGS decomposition theorem. There is no need to choose these error vectors in advance. We emphasize that $x_i'$ and $x_i^+$
 obtained from \eqref{xpimrei} and \eqref{prox-nT}
  should be viewed as  approximate solutions to the minimization problems without the terms involving $\delta_i'$ and $\delta_i$. Once these approximate solutions have been computed, they would generate  $ \delta_i'$ and $ \delta_i$ automatically.   With these known error vectors, we know that the computed approximate solutions are
   the exact solutions to the
minimization problems in \eqref{xpimrei} and \eqref{prox-nT}.

The following proposition is useful in estimating the error term $\Delta(\bdelta',\bdelta)$ in
\eqref{prox-T}.

\begin{proposition} \label{prop-error}
  Denote $\widehat{\cQ}:=\cQ + \cT_{\cQ}$, which is positive definite. Let $\xi = \norm{\widehat\cQ^{-1/2}\Delta(\bdelta',\bdelta)}$. It holds that
  \[\xi \le \norm{\cD^{-1/2}(\bdelta - \bdelta^\prime)} + \norm{\widehat{\cQ}^{-1/2}\bdelta^\prime}.\]
\end{proposition}
\begin{proof}
Recall that $\widehat\cQ = (\cD + \cU)\cD^{-1}(\cD + \cU^*)$. Thus, we have
\[\widehat\cQ^{-1} = (\cD + \cU^*)^{-1}\cD (\cD + \cU)^{-1} = (\cD + \cU^*)^{-1}\cD^{1/2}
\cD^{1/2}(\cD + \cU)^{-1}, \]
which, together with the definition of $\Delta(\bdelta^\prime,\bdelta)$ in \eqref{eq-Delta}, implies that
\begin{align*}
\xi ={}& \norm{\cD^{1/2}(\cD + \cU)^{-1}\bdelta^\prime + \cD^{-1/2}(\bdelta - \bdelta^\prime)}
 \le{} \norm{\cD^{1/2}(\cD + \cU)^{-1}\bdelta^\prime} + \norm{\cD^{-1/2}(\bdelta - \bdelta^\prime)}.
\end{align*}
The desired result then follows.
\end{proof}

Theorem \ref{thm:nbsGS} shows that instead of solving the QP subproblem \eqref{prox-T} directly with
an $N$-dimensional variable $\x$, where $N=\sum_{i=1}^s n_i$, the computation can be decomposed into
$s$ pieces of smaller dimensional problems involving only the variable $x_i$ for each $i=1,\ldots,s$.
Such a decomposition is obviously highly useful for dealing with a large scale CCQP of the form
\eqref{eq-QP} when $N$ is very large. The benefit is
especially important because the computation of $x_i$ for $i=2,\ldots,s$ involves only solving
linear systems of equations.
Of course, one would still have to solve a potentially difficult subproblem involving the
variable $x_1$ due to the presence of the possibly nonsmooth term $p(x_1)$, i.e.,
\begin{eqnarray*}
 x_1^+ &=& \mbox{argmin} \Big\{  p(x_1) + \frac{1}{2}\inprod{x_1}{Q_{11} x_1} - \inprod{c_1}{x_1}  \mid
 x_1\in \cX_1
 \Big\},
\end{eqnarray*}
where $c_1$ is a known vector depending on the previously computed $x_s^\prime,\ldots,x_2^\prime$.
However, in many applications, $p(x_1)$ is usually a simple nonsmooth function
such as $\norm{x_1}_1$, $\norm{x_1}_\infty$, or $\delta_{\R^{n_1}_+}(x_1)$
for which the corresponding
subproblem is not difficult to solve. As a concrete example, suppose that
$Q_{11} = I_{n_1}$.  Then
$x_1^+ = {\rm Prox}_p(c_1)$ and the Moreau-Yosida proximal map ${\rm Prox}_p(c_1)$  can be computed efficiently
for various nonsmooth function $p(\cdot)$ including the examples just mentioned.
In fact, one can always make the subproblem
easier to solve by (a) adding an additional proximal term
$\frac{1}{2}\norm{x_1-\bar{x}_1}^2_{J_1}$ to \eqref{prox-T}, where
$J_1 =  \mu_1 I_{n_1} - Q_{11}$ with $\mu_1=\norm{Q_{11}}_2$; and (b) modifying the sGS operator
to $\cU\widehat{\cD}^{-1}\cU^*$, where $\widehat{\cD} = \cD + \diag{J_1,0,\ldots,0}$.
With the additional proximal term involving $J_1$, the subproblem corresponding to $x_1$ then becomes
\begin{eqnarray*}
 x_1^+ &=& \mbox{argmin} \Big\{  p(x_1) + \frac{1}{2}\inprod{x_1}{Q_{11} x_1} - \inprod{c_1}{x_1}
  + \frac{1}{2}\norm{x_1-\bar{x}_1}_{J_1}^2 \mid x_1\in \R^{n_1}
 \Big\}
 \\[5pt]
 &=& \mbox{argmin} \Big\{  p(x_1) + \frac{\mu_1}{2}\inprod{x_1}{x_1} - \inprod{c_1 + J_1 \bar{x}_1}{x_1}
  \mid x_1\in \R^{n_1}
 \Big\}
 \\[5pt]
 &=& {\rm Prox}_{ p/\mu_1}\big(\mu_1^{-1}(c_1+J_1\bar{x}_1)\big).
\end{eqnarray*}

In fact, more generally, one can also modify the other diagonal blocks in $\cD$ to make
the linear systems involved easier to solve by adding the proximal term
$\frac{1}{2}\norm{\x-\bar{\x}}^2_{\diag{J_1,J_2,\ldots,J_s}}$ to \eqref{prox-T}, where
$J_i\succeq 0$, $i=1,\ldots,s$ are given symmetric matrices. Correspondingly, the sGS linear operator for the
proximal term
to be added to
the problem \eqref{prox-T} then becomes
$\cT_{\cQ+\diag{J_1,\ldots,J_s}} = \cU\widehat{\cD}^{-1}\cU^*,$ where $\widehat{\cD} = \cD + \diag{J_1,J_2,\ldots,J_s}$,
and $\widehat{\cQ}$ in \eqref{psdeq} becomes
$\widehat{\cQ} = \cQ + \diag{J_1,\ldots,J_s} + \cU\widehat{\cD}^{-1}\cU^*$.
There are many suitable choices for $J_i$, $i=2,\ldots,s.$ A conservative choice would be
$J_i = \norm{Q_{ii}}_2 I_{n_i} - Q_{ii}$, in which case the linear system to be solved has its coefficient matrix
given by $Q_{ii} + J_i = \norm{Q_{ii}}_2 I_{n_i}$. Another possible choice of $J_i$ is the sGS linear operator
associated with the matrix $Q_{ii}$, in which case the linear system involved has its coefficient matrix
given by $Q_{ii} + \cT_{Q_{ii}}$ and its solution can be computed by using one cycle of the
sGS method. The latter choice has been considered in \cite{BDY} for its variant of the classical block sGS method. 
Despite the advantage of simplifying the linear systems to be solved, one should
note that the price to pay for adding the extra proximal term
$\frac{1}{2}\norm{\x-\bar{\x}}^2_{\diag{J_1,\ldots,J_s}}$ is worsening
the convergence rate of the overall block sGS method.

\section{A factorization view of the block sGS decomposition
theorem and its equivalence to the SCB reduction procedure}
\label{SCBandsGS}

In this section, we present  a factorization view of the  block sGS decomposition theorem
and show its equivalence to the Schur complement based (SCB) reduction procedure developed in \cite{SCBADMM,LiThesis2014}.

Let $\cO_1$ be the zero matrix in $\R^{n_1\times n_1}$ and $N_1 := n_1$.
For $j=2,\ldots,s$, let $N_j:= \sum_{i=1}^{j} n_i$ and  define $\widehat \cO_j \in \R^{N_{j-1} \times N_{j-1}}$ and $\cO_j \in \R^{N_j\times N_j}$ as follows:
\[\widehat \cO_j := \left[ \begin{array}{c}
 Q_{1,j} \\ \vdots \\ Q_{j-1,j}
\end{array}\right] Q_{j,j}^{-1}\; \left[ Q_{1,j}^*, \; \dots,\; Q_{j-1,j}^*\right]
\]
and
\begin{equation}
\label{eq_Oj}
\cO_{j} :=  \left[ \begin{array}{c}
 Q_{1,2} \\ 0 \\ \vdots \\ 0
\end{array}\right] Q_{2,2}^{-1}\; \left[ Q_{1,2}^*, \; 0,\; \dots,\; 0\right] + \dots +
\left[ \begin{array}{c}
 Q_{1,j} \\ \vdots \\ Q_{j-1,j}\\ 0
\end{array}\right] Q_{j,j}^{-1}\; \left[ Q_{1,j}^*, \; \dots,\; Q_{j-1,j}^*,\; 0\right].
\end{equation}
Then,  the above definitions indicate that, for $2\le j \le s$,
\begin{equation}
\label{eq:re_coj}
\cO_j = {\rm diag}(\cO_{j-1},0_{n_j}) + {\rm diag}(\widehat \cO_j, 0_{n_j}) \in \R^{N_j\times N_j}.
\end{equation}
In \cite{SCBADMM,LiThesis2014}, the SCB reduction procedure corresponding to problem \eqref{eq-QP} is derived through the construction of the above self-adjoint linear operator $ \cO_s$ on $\cX$.
Now we recall the key steps in the SCB reduction procedure derived in the previous work.
For $j=1,\ldots,s$, define
\[
\cQ_{j} := \left[ \begin{array}{cccc}
 Q_{1,1} &\dots & Q_{1,j-1} & Q_{1,j}  \\[5pt]
 \vdots           &\ddots         & \vdots                & \vdots \\[5pt]
Q_{1,j-1}^* &\dots & Q_{j-1,j-1} & Q_{j-1,j}  \\[5pt]
Q_{1,j}^* &\dots & Q_{j-1,j}^* & Q_{j,j}
\end{array}\right], \quad
R_j := \left[\begin{array}{c} Q_{1,j} \\[5pt] \vdots \\ Q_{j-1,j}\end{array}\right].
\]
Note that $\widehat{\Theta}_j = R_j Q_{j,j}^{-1} R_j^*$.
It is easy to show that
\begin{eqnarray}
&& \hspace{-0.7cm}
\min\Big\{ p(x_1) + q(x_{\leq s-1}; x_s) + \frac{1}{2}\norm{\x - \bar{\x}}_{\Theta_s}^2\mid \x\in\cX\Big\}
\label{eq-SCB-1} \\[5pt]
&=&
\min_{x_{\leq s-1} \in \cX_{\leq s-1}} \left\{
\begin{array}{l}p(x_1) + \frac{1}{2}\inprod{x_{\leq s-1}}{\cQ_{s-1}x_{\leq s-1}} - \inprod{b_{\leq s-1}}{x_{\leq s-1}}
 + \frac{1}{2}\norm{x_{\leq s-1} - \bar{x}_{\leq s-1}}_{\Theta_{s-1}}^2
 \\[5pt]
+ \min_{x_s\in\cX_s} \Big\{\frac{1}{2}\inprod{x_s}{Q_{s,s}x_s}  -\inprod{b_s - R_s^* \,x_{\leq s-1}}{x_s}
+ \frac{1}{2} \norm{x_{\leq s-1}-\bar{x}_{\leq s-1}}_{\widehat{\cO}_{s}}^2\Big\}
 \end{array}
 \right\}.
\nn
\end{eqnarray}
By first solving the inner minimization problem with respect to $x_s$, we get the solution as a function
of $x_1,\ldots,x_{s-1}$ as follows:
\begin{eqnarray}
 x_s = Q_{s,s}^{-1} \big(b_s - R_s^* \, x_{\leq s-1}\big).
 \label{eq-xs}
\end{eqnarray}
And the minimum value is given by
\begin{eqnarray*}
& & -\frac{1}{2}\inprod{b_s - R_s^*\, x_{\leq s-1}}{Q_{s,s}^{-1} (b_s - R_s^* x_{\leq s-1})}  + \frac{1}{2} \norm{x_{\leq s-1}-\bar{x}_{\leq s-1}}_{\widehat{\cO}_{s}}^2
 \\[5pt]
 &=& -\frac{1}{2}\inprod{b_s}{Q_{s,s}b_s} +
 \frac{1}{2}\inprod{\bar{x}_{\leq s-1}}{\widehat{\cO}_s \bar{x}_{\leq s-1}}
 +\inprod{R_s Q_{s,s}^{-1} (b_s-R_s^*\bar{x}_{\leq s-1})}{x_{\leq s-1}}.
\end{eqnarray*}
Thus \eqref{eq-SCB-1} reduces to a problem involving only the variables
$x_1,\ldots,x_{s-1}$, which, up to a constant, is given by
\begin{eqnarray}
&& \hspace{-0.7cm} \min_{x_{\leq s-1} \in \cX_{\leq s-1}} \left\{
 \begin{array}{l}
p(x_1) + \frac{1}{2}\inprod{x_{\leq s-1}}{\cQ_{s-1}x_{\leq s-1}} -
\inprod{b_{\leq s-1} - R_s x_s^\prime}{x_{\leq s-1}}
\\[5pt]
 + \frac{1}{2}\norm{x_{\leq s-1} - \bar{x}_{\leq s-1}}_{\Theta_{s-1}}^2
 \end{array}
 \right\}
 \nn
 \\[5pt]
 &=& \min_{x_{\leq s-1} \in \cX_{\leq s-1}} \left\{
 \begin{array}{l}
p(x_1) + q(x_{\leq s-1}; x_s^\prime)
 + \frac{1}{2}\norm{x_{\leq s-1} - \bar{x}_{\leq s-1}}_{\Theta_{s-1}}^2
 \end{array}
 \right\},
\label{eq-SCB-2}
\end{eqnarray}
where $x_s^\prime = Q_{s,s}^{-1} (b_s-R_s^*\bar{x}_{\leq s-1}).$
Observe that \eqref{eq-SCB-2} has exactly the same form as \eqref{eq-SCB-1}.
By repeating the above procedure to sequentially eliminate the variables $x_{s-1},\ldots, x_2$, we will finally
arrive at a minimization problem involving only the variable $x_1$. Once that minimization problem
is solved, we can recover the solutions for $x_2,\ldots, x_s$ in a sequential manner.

Now we will prove the equivalence between the block sGS decomposition theorem
 and the SCB reduction procedure in the subsequent analysis by proving that $\cO_s = \cT_\cQ$,
where $\cT_\cQ$ is given in \eqref{eq-Tsgs}.
For $j=2,\dots,s$, define the block matrices $\widehat \cV_j\in \R^{N_j\times N_j}$ and $\cV_j \in\R^{N\times N}$ by
\begin{equation}\label{eq:def_cUj}
 \widehat \cV_j := \left[ \begin{array}{cccc}
  I_{n_1}                       &     &  &    Q_{1,j}Q_{j,j}^{-1}   \\
  & \ddots            &    & \vdots    \\
          &      &I_{n_{j-1}}  & Q_{j-1,j}Q_{j,j}^{-1}    \\[5pt]
     &     &  & I_{n_j}    \\[5pt]
\end{array}\right] \in \R^{N_j\times N_j}, \quad \cV_j := {\rm diag}(\widehat \cV_j, I_{N-N_j}) \in \R^{N\times N},
\end{equation}
where $I_{n_j}$ is the $n_j\times n_j$ identity matrix. Note that $\widehat \cV_s = \cV_s$.
Given $j\ge 2$,
 we have, by simple calculations, that for any $k<j$,
\begin{equation}\label{eq-ujq}
 \widehat \cV_j^{-1} \left[ \begin{array}{c}
 Q_{1,k}\\ \vdots \\ Q_{k-1,k}\\[5pt] 0
\end{array}\right]
= \left[ \begin{array}{c}
 Q_{1,k}\\ \vdots \\ Q_{k-1,k}\\[5pt] 0
\end{array}\right]
\quad
{\rm and}
\quad
\widehat \cV_j^{-1} \left[ \begin{array}{c}
 Q_{1,j}\\ \vdots \\ Q_{j-1,j}\\[5pt] Q_{j,j}
\end{array}\right]
=  \left[ \begin{array}{c}
  0\\ \vdots \\   0 \\[5pt] Q_{j,j}
\end{array}\right].
\end{equation}
From \eqref{eq_Oj} and \eqref{eq-ujq}, we have that
\begin{equation}
  \label{eq-ujojuj}
  \widehat \cV_j^{-1}\cO_j     (  \widehat   \cV_j ^{-1} )^* = \cO_j,\quad j=2,\ldots,s.
\end{equation}

\begin{lemma} \label{lem-SGS-1}
Let $\cU$ and $\cD$ be given in \eqref{eq-UD}. It holds that
\begin{eqnarray*}
  & \cV_2^* \cdots \cV_s^* = \cD^{-1} (\cD+ \cU^*), \quad
\cV_s \cdots \cV_2 = (\cD+ \cU) \cD^{-1}. &
\end{eqnarray*}
\end{lemma}
\begin{proof} It can be verified directly that
\begin{eqnarray*}
 \cV_2^* \cdots \cV_s^* = \left[\begin{array}{ccccc}
   I & & & &  \\[5pt]
  Q_{2,2}^{-1}Q_{1,2}^*  & I    & & &  \\
  \vdots                               &\ddots  &\ddots & & \\[5pt]
  Q_{s,s}^{-1} Q_{1,s}^*    &\cdots     & Q_{s,s}^{-1} Q_{s-1,s}^* & I
\end{array}\right]
=  \cD^{-1} (\cD+ \cU^*).
\end{eqnarray*}
The second equality follows readily from the first.
\end{proof}

In the proof of the next lemma, we will make use of the well known fact that
for given  symmetric matrices   $A,C$ such that $C\succ 0$ and $M:= A-BC^{-1}B^* \succ 0$,
we have that
\begin{eqnarray}
\left[ \begin{array}{cc} A & B \\ B^* & C \end{array}\right] =
\left[ \begin{array}{cc} I & B C^{-1}\\ 0 & I \end{array}\right]
\left[ \begin{array}{cc} M & 0 \\ 0 & C \end{array}\right]
\left[ \begin{array}{cc} I & 0 \\ C^{-1}B^* & I \end{array}\right].
\label{eq-schur}
\end{eqnarray}

\begin{theorem}  \label{lem-SGS-2}
It holds that
\[
 \cQ_s + \cO_s
 =
\cV_s\cdots \cV_2 \, \cD \,
\cV_2^* \cdots \cV_s^*
\quad {\rm and} \quad  \cO_s = \cT_\cQ.
\]
\end{theorem}
\begin{proof} By using \eqref{eq-schur}, for $j=2,\ldots, s$, we have that
\begin{equation*}
  \cQ_j = \widehat\cV_j \, {\rm diag}(\cM_{j - 1},Q_{j,j}) \, \widehat\cV_j^*,
\end{equation*}
where
\begin{eqnarray*}
\cM_{j-1} = \left[ \begin{array}{ccc}
  Q_{1,1} &  \dots                       & Q_{1,j-1}  \\
  \vdots & \ddots  & \vdots               \\
 Q_{1,j-1}^*  &\dots        & Q_{j-1,j-1}
\end{array}\right] -
 \left[ \begin{array}{c}
Q_{1,j} \\ \vdots  \\ Q_{j-1,j}
\end{array}\right] Q_{j,j}^{-1}  \left[ \begin{array}{c}
Q_{1,j}^*,\; \dots,\;   Q_{j-1,j}^*
\end{array}\right]
\;=\; \cQ_{j-1}- \widehat \cO_{j}.
\end{eqnarray*}
Thus, from \eqref{eq-ujojuj}, we know that for $2\le j \le s$,
\begin{eqnarray*}
 \cQ_j+\cO_{j} = \widehat \cV_j \Big(\diag{\cM_{j-1},Q_{j,j}}      + \widehat \cV_j^{-1}\cO_{j}     (  \widehat   \cV_j ^{-1} )^*
\Big) \widehat \cV_j^* \;=\; \widehat \cV_j \Big(\diag{\cM_{j-1},Q_{j,j}}      +\cO_{j}
\Big) \widehat \cV_j^*.
\end{eqnarray*}
For $2\le j \le s$, by \eqref{eq:re_coj}, we have that
$$
 \diag{\cM_{j-1},Q_{j,j}} + \cO_{j} = \diag{\cM_{j-1}+\cO_{j-1}+\widehat \cO_{j},Q_{j,j}}
=
\diag{\cQ_{j-1}+\cO_{j-1},Q_{j,j}}
$$
and consequently,
\begin{equation}\label{eq:cQjpcOj}
\cQ_j + \cO_j = \widehat \cV_j\, {\rm diag}(\cQ_{j-1} + \cO_{j-1},Q_{j,j})\,\widehat \cV_j^*.
\end{equation}
Thus, by recalling the definitions of $\widehat \cV_j$ and $\cV_j$ in \eqref{eq:def_cUj} and using \eqref{eq:cQjpcOj}, we obtain through simple calculations that
\begin{align*}
 \cQ_s+\cO_{s} ={}& \cV_s \;\diag{\cQ_{s-1}+\cO_{s-1},Q_{s,s}}\;  \cV_s^*
\\
={}& \vdots
\\
={}& \cV_s\cdots \cV_2\;\diag{\cQ_1 + \cO_1,Q_{2,2},\dots,Q_{s,s}}\;
\cV_2^* \cdots \cV_s^*.
\end{align*}
Thus, by using the  fact that  $\cQ_1 + \cO_1 = Q_{1,1}$, we get
\[
 \cQ_s + \cO_s
 =
\cV_s\cdots \cV_2\;{\rm diag}(Q_{1,1},Q_{2,2},\dots,Q_{s,s})\;
\cV_2^* \cdots \cV_s^*.\]
  By Lemma \ref{lem-SGS-1},
it follows that
\begin{align*}
 \cQ_s + \cO_s ={}& (\cD+\cU) \cD^{-1} \cD \cD^{-1}(\cD + \cU^*)
\;=\; (\cD+\cU)  \cD^{-1}(\cD + \cU^*)
\;=\; \cQ + \cT_\cQ,
\end{align*}
where the last equation follows from \eqref{psdeq} in Theorem \ref{thm:nbsGS}. Since $\cQ_s = \cQ$, we know that \begin{eqnarray*}
\cO_s = \cT_\cQ.
\end{eqnarray*}
This completes the proof of the theorem.
\end{proof}

\section{An extended block sGS method for solving the CCQP \eqref{eq-QP}}
\label{sec-pg}

With the block sGS decomposition theorem (Theorem \ref{thm:nbsGS}) and Proposition \ref{prop-error},
we can now extend the classical block sGS method to
solve
the CCQP \eqref{eq-QP}.
The detail steps of the algorithm for  solving \eqref{eq-QP} are given as follows.

\bigskip
\centerline{\fbox{\parbox{\textwidth}{
			{\bf Algorithm 1: An sGS based inexact proximal gradient method for \eqref{eq-QP}}.
\\[5pt]
Input $\widetilde \x^1 = \x^0\in{\rm dom}(p)\times\R^{n_2}\times\ldots\times\R^{n_s}$, $t_1 = 1$ and a summable sequence of nonnegative numbers $\{\eps_k\}$.
For $k=1,2,\ldots$, perform the following steps in each iteration.
			\begin{description}				
				\item[Step 1.]
Compute
\begin{equation}
\x^k = \mbox{argmin}_{\x\in\cX}\; \Big\{p(x_1)+q(\x)
  +\frac{1}{2}\norm{\x  - \widetilde \x^k }_{\cT_\cQ}^2 -\inprod{\x}{\Delta(\, \widetilde \bdelta^k, \bdelta^k)}  \Big\},
  \label{eq-apg}
\end{equation}
via the  sGS decomposition procedure described in Theorem \ref{thm:nbsGS},
where
$\widetilde \bdelta^k, \,  \bdelta^k \in \cX$ are error vectors such that
				\begin{equation}
				\max \{ \norm{\widetilde\bdelta^k}, \norm{\bdelta^k}\}\leq \frac{\eps_k}{ t_k}.
				\label{eq-error}
				\end{equation}
				\item [Step 2.] Choose $t_{k+1}$ such that $ t_{k+1}^2-t_{k+1}\leq t_k^2 $ and set $\beta_k = \frac{t_k-1}{t_{k+1}}$. Compute
				\[\widetilde \x^{k+1} = \x^k + \beta_k (\x^k - \x^{k-1}).\]					
				\end{description}
				}}}

\bigskip

We have the following iteration complexity convergence results for Algorithm 1.
\begin{proposition}
  \label{prop:convergence_sgs}
  Suppose $\x^*$ is an optimal solution of problem  \eqref{eq-QP}. Let $\{\x^k\}$ be the sequence generated by Algorithm 1. Define  $M =2 \norm{\cD^{-1/2}}_2 + \norm{\widehat{\cQ}^{-1/2}}_2$.
  \\
 (a)  If $t_{k+1} = \frac{1 + \sqrt{1 + 4t_k^2}}{2}$ for all $k\ge 1$, it holds that
\begin{eqnarray*}
  0\leq  F(\x^k) - F(\x^*) \leq \frac{2}{(k+1)^2}\Big(\norm{\x^0-\x^*}_{\widehat{\cQ}}+\bar{\eps}_k\Big)^2,
\end{eqnarray*}
where $\widehat{\cQ} = \cQ + \cT_\cQ$ and $\bar{\eps}_k = 2 M \sum_{i=1}^k \eps_i$.
\\[5pt]
  (b) If $t_k =  1$ for all $k\ge 1$, it holds that
  \begin{eqnarray*}
  0\leq F(\x^k) -  F(\x^*) \leq  \frac{1}{2k} \Big( \norm{\x^0 - \x^*}_\hcQ + \tilde{\eps}_k\Big)^2,
  \end{eqnarray*}
  where $\tilde{\eps}_k = 4M \sum_{i=1}^k i \eps_i$.
\end{proposition}
\begin{proof}
(a)
The result can be proved by applying Theorem 2.1 in \cite{JST12}. In order to  apply the theorem,
we need to verify that the error $\e:= \mbox{\boldmath{$\gamma$}} + \cQ \x^k -\b +\cT_\cQ (\x^k-\widetilde{\x}^k) $, where
$\mbox{\boldmath{$\gamma$}} = (\gamma_1; 0; \ldots; 0)$ and
$\gamma_1 \in \partial p(x_1^k)$,  incurred for solving the subproblem (without the
perturbation term $\Delta(\, \widetilde \bdelta^k,
 \bdelta^k)$)  in Step 1  inexactly
is sufficiently small. From Theorem \ref{thm:nbsGS}, we know that
\begin{eqnarray*}
 \e :=  \mbox{\boldmath{$\gamma$}} + \cQ \x^k -\b +\cT_\cQ (\x^k-\widetilde{\x}^k) = \Delta(\, \widetilde \bdelta^k,
  \bdelta^k).
\end{eqnarray*}
The theorem is proved via Theorem 2.1 in \cite{JST12} if we can show that $
  \norm{\widehat{\cQ}^{-1/2}\Delta(\, \widetilde \bdelta^k,
 \bdelta^k)} \leq M \frac{\epsilon_k}{t_k}$. But from \eqref{eq-error} and Proposition \ref{prop-error}, we have that
 \begin{eqnarray*}
    \norm{\widehat{\cQ}^{-1/2}\Delta(\, \widetilde \bdelta^k,
 \bdelta^k)}
 \leq  \norm{\cD^{-1/2} \bdelta^k} + \norm{\cD^{-1/2}\widetilde \bdelta^k}  + \norm{\widehat{\cQ}^{-1/2}\widetilde \bdelta^k}
  \leq M \eps_k/t_k,
 \end{eqnarray*}
 thus
  the required inequality
  indeed holds true, and the proof is completed.
  \\[5pt]
  (b) There is no straightforward theorem for which we can apply to prove the result, we will provide the
  proof in the Appendix.
\end{proof}

\begin{remark}
  \label{rmk:sgs_opt}
  It is not difficult to show that if  $p(\cdot) \equiv 0$, $t_k=1$, and
  $\bdelta^k = \widetilde \bdelta^k = 0$ for all $k\ge 1$,
  then Algorithm 1 exactly coincides with the classical block sGS method
  \eqref{eq-SGS}; and if $\bdelta^k $, $\widetilde \bdelta^k$ are  allowed to be non-zero but
  satisfy the condition \eqref{eq-error} for all $k\ge 1$, then we obtain the inexact extension of the
  classical block sGS method. 
 \end{remark}
 \begin{remark}
  Proposition \ref{prop:convergence_sgs} shows
  that the classical block sGS method for solving \eqref{eq-1} can be extended to solve the convex composite
  QP \eqref{eq-QP}. It also demonstrates the advantage of interpreting
  the block sGS method from the optimization perspective. For example, one can obtain the  $O(1/k)$ iteration complexity
  result for the classical block sGS method without assuming that $\cQ$ is positive definite.
   To the best of our knowledge, such a complexity result for the classical block sGS is new.
   More importantly, inexact and accelerated versions of the block sGS method can also be derived for \eqref{eq-1}.
\end{remark}
\begin{remark}
In solving  \eqref{eq-apg} via the sGS decomposition procedure to satisfy the error condition
\eqref{eq-error}, let $\x^\prime = [x'_1;\ldots;x'_s]$ be the intermediate solution computed during the
backward GS sweep (in Theorem \ref{thm:nbsGS})
and the associated error vector be $\widetilde{\bdelta}^k=[\widetilde{\delta}^k_1;\ldots;
\widetilde{\delta}^k_s]$. In the forward GS sweep,
one can often save computations
by using the computed $x^\prime_i$ to estimate $x^{k+1}_i$ for $i\geq 2$, and the
resulting error vector will be given by
$
\delta^{k}_i = \widetilde{\delta}^k_i + \mbox{$\sum_{j=1}^{i-1}$} Q_{ji}^* (x_j^{k+1}-\widetilde{x}^k_j).
$
If we have that
\begin{eqnarray}
\norm{\mbox{$\sum_{j=1}^{i-1}$} Q_{ji}^* (x_j^{k+1}-\widetilde{x}^k_j)} \;\leq\; \rho ,
\label{eq-err2}
\end{eqnarray}
where $\rho = \frac{c}{\sqrt{s}}\norm{\widetilde{\bdelta}^k}$ and
 $c > 0$ is some given  constant,
 then clearly
$\norm{\delta^k_i}^2 \leq  2 \norm{\widetilde{\delta}^k_i}^2 + 2\rho^2 $.
When all the error components  $\norm{\delta^k_i}^2$ satisfy the previous bound for $i=1,\ldots,s$,
regardless of whether $x^{k+1}_i$ is estimated from $x^\prime_i$ or computed afresh,  we get
$\norm{\bdelta^k} \leq \sqrt{2(1+c^2)}\norm{\widetilde{\bdelta}^k}$.
Consequently the error condition  \eqref{eq-error} can be satisfied with a slightly larger error
tolerance $\sqrt{2(1+c^2)}\,\eps_k/t_k $. It is easy to see that one can use the
condition in \eqref{eq-err2} to decide whether $x^{k+1}_i$ can be estimated from $x^\prime_i$
without contributing a large error to $\norm{\bdelta^k}$
for each $i=2,\ldots,s$.
\end{remark}

 {Besides the above iteration complexity results, one can also study the linear convergence rate of Algorithm 1. Indeed,
just} as in the case of the classical block sGS method, the convergence rate
of our extended inexact block sGS method
 for solving \eqref{eq-QP} can also be established
when $\cQ \succ 0$. The precise result is given in the next theorem.

\begin{theorem} \label{thm:convergence_rate}
Suppose that the relative interior of
the domain of $p$, ${\rm ri}({\rm dom}(p))$,  is non-empty,  
$\cQ \succ 0$ and $t_k=1$ for all $k\geq 1$.  Then
\begin{eqnarray}
  \norm{\hcQ^{-1/2}(\x^k-\x^*)} \leq \norm{\cB}_2^k \norm{\hcQ^{-1/2}(\x^0-\x^*)} +  M \norm{\cB}_2^k
  \sum_{j=1}^k \norm{\cB}_2^{-j} \eps_j,
  \label{eq-rate}
\end{eqnarray}
where $\cB = I- \hcQ^{-1/2} \cQ \hcQ^{-1/2}$,
and $M$ is defined as in Proposition
\ref{prop:convergence_sgs}.
Note that $0\preceq  \cB \prec I$.
\end{theorem}
\begin{proof}
For notational convenience, we let $\Delta^j = \Delta(\widetilde{\bdelta}^j,\bdelta^j)$ in this proof.

Define $E_1: \cX \to \cX_1$ by
$E_1(\x) = x_1$ and $\widehat{p}: \cX \to (-\infty,\infty]$ by $\widehat{p}(\x) = p(E_1\hcQ^{-1/2} \x)$.
 Since $\hcQ \succ 0$, it is clear that ${\rm range}(E_1\hcQ^{-1/2}) =\cX_1$ and
hence ${\rm ri}({\rm dom}(p))\cap {\rm range}(E_1\hcQ^{-1/2})  \not= \emptyset$.
By \cite[Theorem 23.9]{Rockafellar}, we have that
\begin{eqnarray}
 \partial \widehat{p} (\x) = \hcQ^{-1/2} E_1^* \partial p (E_1\hcQ^{-1/2} \x) \quad \forall\; \x \in \cX.
 \label{eq-partial}
\end{eqnarray}
From the optimality condition of $\x^j$, we have that
\begin{eqnarray*}
  && 0 =  \bgamma^j + \hcQ (\x^j-\x^{j-1}) -\b + \cQ \x^{j-1} - \Delta^j
  \\[5pt]
\Leftrightarrow &&  \hcQ^{1/2}\x^{j-1} + \hcQ^{-1/2}(\b -\cQ \x^{j-1}) + \hcQ^{-1/2} \Delta^j
 = \hcQ^{-1/2}\bgamma^j + \hcQ^{1/2} \x^j,
\end{eqnarray*}
where $\bgamma^j = (\gam_1^j;0;\ldots;0)$ with $\gam_1^j \in \partial p(x_1^j)$.
Let $\hat{\x}^j = \hcQ^{1/2}\x^j$ and $\hat{\x}^{j-1} = \hcQ^{1/2}\x^{j-1}$. Then we have that
\begin{eqnarray*}
 &&  \hat{\x}^{j-1}  + \hcQ^{-1/2}(\b -\cQ \x^{j-1}) + \hcQ^{-1/2} \Delta^j
   \in (I + \partial \widehat{p})( \hat{\x}^j)
   \\[5pt]
   \Leftrightarrow && \hat{\x}^j = {\rm Prox}_{\widehat{p}}\big( \cB \hat{\x}^{j-1}
    + \hcQ^{-1/2}\b  + \hcQ^{-1/2} \Delta^j \big) .
\end{eqnarray*}
Similarly if $\x^*$ is an optimal solution of \eqref{eq-QP}, then we have that
\begin{eqnarray*}
  \hat{\x}^* = {\rm Prox}_{\widehat{p}} (\cB\hat{\x}^*  + \hcQ^{-1/2}\b).
\end{eqnarray*}
By using the nonexpansive property of ${\rm Prox}_{\widehat{p}}$, we have that
\begin{eqnarray*}
 \norm{\hat{\x}^j - \hat{\x}^*} &=& \norm{ {\rm Prox}_{\widehat{p}}\big( \cB \hat{\x}^{j-1}
    + \hcQ^{-1/2}\b  + \hcQ^{-1/2} \Delta^j \big) -
    {\rm Prox}_{\widehat{p}} (\cB\hat{\x}^*  + \hcQ^{-1/2}\b)}
    \\[5pt]
    &\leq &
    \norm{ ( \cB \hat{\x}^{j-1}
    + \hcQ^{-1/2}\b  + \hcQ^{-1/2} \Delta^j ) -
   (\cB\hat{\x}^*  + \hcQ^{-1/2}\b)}
   \\[5pt]
   &\leq & \norm{\cB}_2\norm{\hat{\x}^{j-1}-\hat{\x}^*} + M\eps_j.
\end{eqnarray*}
By applying the above inequality sequentially for $j=k,k-1,\ldots,1$, we get the required result in
\eqref{eq-rate}.
\end{proof}

\begin{remark}
	\label{rmk:convergence_sgs}
	 {In fact, one can weaken the positive definiteness assumption of $\cQ$ in the above theorem and still expect a linear rate of convergence. As a simple illustration,
	we only discuss here the exact version of Algorithm 1, i.e., $\widetilde \bdelta^k = \bdelta^k = 0$, under the error bound condition \cite{Luo1992on,Luo1993error} on $F$ which  holds automatically if $p$ is a convex piecewise quadratic/linear function such as
	$p(x_1) =\norm{x_1}_1$, $p(x_1) =  \delta_{\R^{n_1}_+}$ or if $\cQ \succ 0$. When $t_k = 1$ for all $k\ge 1$, one can prove that $\{F(\x^k) \}$ converges at least Q-linearly and
	$\{\x^k\}$ converges at least R-linearly to an optimal solution of  problem  \eqref{eq-QP} by using the techniques developed in \cite{Luo1992on,Luo1993error}.
	Interested readers may refer to \cite{Tseng2009a,Zhou2015a} for more details. For the accelerated case, with the additional fixed restarting scheme incorporated in Algorithm 1, both the R-linear convergences of $\{ F(\x^k)\}$ and $\{ \x^k \}$ can be obtained from \cite[Corollary 3.8]{Wen2017linear}.}
	
\end{remark}

\section{An illustration on the application of the block sGS decomposition theorem in designing
an efficient proximal ALM}
\label{sec-pALM}

In this section, we demonstrate the usefulness of our  block sGS decomposition theorem as
a  building block for designing an efficient proximal ALM for solving a linearly constrained
convex composite QP problem given by
\begin{eqnarray}
 \min \Big\{ p(x_1) + \frac{1}{2}\inprod{\x}{\cP\x} -\inprod{\g}{\x} \mid \cA\x = d \Big\},
 \label{eq-QP2}
\end{eqnarray}
where $\cP$ is a positive semidefinite linear operator on $\cX$, $\cA: \cX \to \cY$ is a given linear map,
and $\g\in\cX$, $d\in\cY$ are given data. Here $\cX$ and $\cY$ are two finite dimensional
inner product spaces.
Specifically, we show how the block sGS decomposition theorem given in
Theorem \ref{thm:nbsGS} can be applied within the proximal ALM.
We must emphasize that our main purpose here is to briefly illustrate the usefulness of
the block sGS decomposition theorem  but not to focus on
the proximal ALM itself.  {Indeed, simply being capable of handling the nonsmooth
function $p(\cdot)$ has already distinguished our approach from other approaches of using the sGS technique in optimization algorithms, e.g, \cite{Kristian2015apreconditioned,Kristian2015bpreconditioned}, where the authors {incorporated the} pointwise sGS splitting as a preconditioner within the Douglas--Rachford splitting  method for a convex-concave saddle point problem.}

In depth analysis of  various  recently developed ADMM-type algorithms
and accelerated block coordinate descent
algorithms employing the
block sGS decomposition theorem as a building block
 can be found in \cite{CST17,DWD,QSDPNAL,SCBADMM,STY16}. Thus we shall not
 elaborate here again  on the essential role played by the block sGS decomposition theorem
 in the design of those algorithms.

Although the problem \eqref{eq-QP2} looks deceivingly simple, in fact it is a powerful model
which
includes the important class of standard convex quadratic semidefinite programming (QSDP) in the dual form
given by
\begin{eqnarray}
\min\Big\{ \frac{1}{2}\inprod{W}{\cH W} -\inprod{h}{\xi} \mid Z +  \cB^* \xi +\cH W = C,\; \xi\in \R^p, \;
Z \in \S^n_+,\; W\in {\cW}  \Big\},
\label{eq-QSDP}
\end{eqnarray}
where $h\in\R^p$, $C\in\S^n$ are given data, $\cB:\S^n\to\R^p$ is a given linear map that is assumed to be surjective,
$\cH: \S^n \to \S^n$ is a self-adjoint positive semidefinite linear operator, and $\cW \subseteq \S^n$ is any subspace containing $ {\rm Range}(\cH)$,  the range space of $\cH$.
Here $\S^n$ denotes the space of $n\times n$  symmetric matrices and $\S^n_+$ denotes
the cone of symmetric positive semidefinite matrices in $\S^n$.
One can obviously express the QSDP problem \eqref{eq-QSDP} in the form of \eqref{eq-QP2} by defining
$\x = (Z;\xi;W)$, $p(Z) = \delta_{\S^n_+}(Z)$,
$\cP=\diag{0,0,\cH}$, and $\cA = (\cI,\cB^*,\cH)$.

We begin with the augmented Lagrangian function associated with
\eqref{eq-QP2}:
\begin{eqnarray}
 L_\sig (\x; y) =  p(x_1) + \frac{1}{2}\inprod{\x}{\cP\x} -\inprod{\g}{\x}  + \frac{\sig}{2}
 \norm{\cA \x- d+\sig^{-1} y}^2 - \frac{1}{2\sig}\norm{y}^2,
\end{eqnarray}
where $\sig > 0$ is a given penalty parameter and $y\in\cY$ is the multiplier associated
with the equality constraint.
 The template for a proximal ALM is given as follows. Given $\cT \succeq 0$, $\x^0\in\cX$ and $y^0\in\cY$.
 Perform the following steps in each iteration.
 \begin{description}
 \item[Step 1.] Compute
 \begin{eqnarray}
&& \hspace{-0.7cm}
  \x^{k+1} = \mbox{argmin} \Big\{L_\sig(\x; y^k) + \frac{1}{2}\norm{\x-\x^k}_\cT^2 \mid \x\in\cX\Big\}
  \nn \\[2pt]
  &=& \mbox{argmin} \Big\{ p(x_1) + \frac{1}{2}\inprod{\x}{(\cP + \sig \cA^*\cA)\x}
  -\inprod{\b}{\x} + \frac{1}{2}\norm{\x-\x^k}_\cT^2 \mid \x\in\cX
  \Big\}, \qquad
  \label{eq-step1}
 \end{eqnarray}
 where $\b = \g +\cA^*(\sig d - y^k).$
 \item[Step 2.] Compute $y^{k+1} = y^k + \tau \sig (\cA\x^k-d)$, where $\tau\in (0,2)$ is
 the  step-length.
 \end{description}
It is clear that the subproblem \eqref{eq-step1} has the form given in \eqref{eq-QP}.
  Thus,   one can apply the block sGS decomposition theorem to efficiently solve the subproblem
if we choose $\cT = \cT_{\cP + \sig \cA^*\cA}$, i.e., the sGS operator
associated with $\cQ: = \cP + \sig \cA^*\cA$.  For the QSDP problem \eqref{eq-QSDP} with $\cW: ={\rm Range}(\cH)$, we have
that
\begin{eqnarray*}
  \cQ = \sig \left( \begin{array}{ccc}
     \cI  & \cB^* & \cH \\[5pt]
     \cB & \cB\cB^* & \cB \cH \\[5pt]
     \cH &\cH\cB^* & \sig^{-1} \cH + \cH^2
  \end{array} \right)
\end{eqnarray*}
and that the subproblem \eqref{eq-step1} can be efficiently solved by one cycle of the extended block sGS method explicitly
as follows, given the iterate $(Z^k,\xi^k,\cH W^k)$ and multiplier $y^k$.
\begin{description}
\item[Step 1a.] Compute $\cH W^\prime$ as the solution of
$
 (\sig^{-1} \cI  + \cH)\cH W^\prime = \sig^{-1}b_W - \cH Z^k - \cH \cB^* \xi^k,
$
where $b_W = \cH (\sig C -  Y^k)$.
\item[Step 1b.] Compute $\xi^\prime$ from
$
 \cB\cB^* \xi^\prime = \sig^{-1} b_\xi - \cB Z^k - \cB \cH W^\prime,
$
where $b_\xi =  h + \cB (\sig C -  Y^k)$.
\item[Step 1c.] Compute
$
Z^{k+1} = \mbox{argmin} \Big\{ \delta_{\S^n_+}(Z) + \frac{\sig}{2} \norm{Z + \cB^*\xi^\prime + \cH W^\prime -  \sig^{-1} b_Z}^2 \Big\},
$
where $b_Z = \sig C -  Y^k$.
\item[Step 1d.] Compute $\xi^{k+1}$ from
$
\cB\cB^* \xi^{k+1} = \sig^{-1} b_\xi - \cB Z^{k+1} - \cB \cH W^\prime.
$
\item[Step 1e.] Compute $\cH W^{k+1}$ from
$(\sig^{-1} \cI  + \cH)\cH W^{k+1} = \sig^{-1} b_W - \cH Z^{k+1} - \cH \cB^* \xi^{k+1}.
$
\end{description}
From the above implementation, one can see how simple it is for one to apply
the block sGS decomposition theorem to solve the
complicated subproblem \eqref{eq-step1} arising from QSDP. Note that in Step 1a and Step 1e, we only need to compute $\cH W^\prime$ and $\cH W^{k+1}$, respectively, and we do not need the values of  $  W^\prime$ and $ W^{k+1}$ explicitly. Here, for simplicity, we only write down the exact version of a proximal ALM by using our exact block sGS decomposition theorem. Without any difficulty, one can also apply the inexact version of the  block sGS decomposition theorem to derive a more practical inexact proximal ALM for solving (\ref{eq-QP2}), say
when the linear systems involved are large scale and have to be solved by a Krylov subspace iterative method.

\section{Extension of the classical block symmetric SOR method for solving \eqref{eq-QP}}
\label{sSOR}

In a   way similar to   what we have done in section \ref{sec:sGS},
we show in this section that  the {classical} block symmetric SOR (block sSOR) method can also be interpreted  from an optimization perspective.

Given a parameter $\ome\in [1,2)$, the $k$th iteration of the {classical} block sSOR method in the third normal form is  defined by
\begin{equation}
  \widehat{\cQ}_\omega (\x^{k+1}-\x^k) =  b -\cQ\x^k,
\label{eq-SSOR}
\end{equation}
where \[\widehat{\cQ}_\omega
 =
 (\tau\cD+\cU)^{-1}(\rho\cD)^{-1}(\tau\cD+\cU^*),\] $\tau = 1/\ome$, and $\rho = 2\tau -1$.
Note that for $\ome\in [1,2)$, we have that $\tau \in (1/2,1]$ and $\rho \in (0,1]$.
We should mention that the {classical} block sSOR method is typically not derived in the form given in \eqref{eq-SSOR}, see for example \cite[p.117]{Hackbusch}, but one can
show with some algebraic manipulations that \eqref{eq-SSOR} is an equivalent reformulation.

Denote
\[ \cT_{\rm sSOR} :=  \big((1-\tau)\cD+\cU\big)(\rho\cD)^{-1} \big( (1-\tau)\cD+\cU^*\big). \]
In the next proposition, we show that $\cW$ can be decomposed as the sum of $\cQ$ and $\cT_{\rm sSOR}$.
Similar to the linear operator $\cT_\cQ$ in section \ref{sec:sGS}, $\cT_{\rm sSOR}$ is the key ingredient
which enables us to derive the block sSOR method from the optimization perspective,
and to extend it to solve the CCQP \eqref{eq-QP}.
\begin{proposition}
Let $\omega \in [1,2)$, and denote $\tau=1/\omega\in(1/2,1]$,  $\rho = 2\tau - 1$.
It holds that \begin{equation}
 \widehat{\cQ}_\omega = \cQ + \cT_{\rm sSOR}.
\label{eq-W}
\end{equation}
\end{proposition}
\begin{proof}
Let $\bar{\tau}:=\tau-\frac{1}{2}>0$ and $\overline{\cU} = \cU+\frac{1}{2}\cD$.
Note that $\rho= 2\bar{\tau}$ and
\begin{align*}
   \widehat{\cQ}_\omega ={}& (\bar{\tau}\cD+\overline{\cU}) (2\bar{\tau}\cD)^{-1}  (\bar{\tau}\cD+\overline{\cU}^*)
\\[5pt]
={}& \frac{1}{2} (\bar{\tau}\cD+\overline{\cU}) ( I + (\bar{\tau}\cD)^{-1}\overline{\cU}^* )
=  \frac{1}{2}(\bar{\tau}\cD+\overline{\cU} + \overline{\cU}^* + \overline{\cU}(\bar{\tau}\cD)^{-1}\overline{\cU}^* )
\\[5pt]
={}&  \frac{1}{2} ( \cQ + \bar{\tau}\cD + \overline{\cU}(\bar{\tau}\cD)^{-1}\overline{\cU}^* )
\\[5pt]
={}& \cQ + \frac{1}{2} \big(\bar{\tau}\cD + \overline{\cU}(\bar{\tau}\cD)^{-1}\overline{\cU}^*-\cQ\big).
\end{align*}
Now
\begin{eqnarray*}
& &\hspace{-0.7cm}
\bar{\tau}\cD + \overline{\cU}(\bar{\tau}\cD)^{-1}\overline{\cU}^*-\cQ \;= \;
\bar{\tau}\cD + \overline{\cU}(\bar{\tau}\cD)^{-1}\overline{\cU}^* -\overline{\cU}-\overline{\cU}^*
\\[5pt]
&=&  ( \bar{\tau}\cD-\overline{\cU} )(\bar{\tau}\cD)^{-1}
  ( \bar{\tau}\cD-\overline{\cU}^* )
\;=\; \big((1-\tau)\cD+\cU\big)(\bar{\tau}\cD)^{-1} \big( (1-\tau)\cD+\cU^*\big).
\end{eqnarray*}
From here, we get the required expression for $ \widehat{\cQ}_\omega$ in \eqref{eq-W}.
\end{proof}


Given two error tolerance vectors $\bdelta$ and $\bdelta'$ with $\delta_1 = \delta'_1$, let
\[\Delta_{\rm sSOR}(\bdelta',\bdelta) := \bdelta' + (\tau \cD + \cU)(\rho \cD)^{-1}(\bdelta - \bdelta').\]
Given $\bar \x\in\cX$, similar to Theorem \ref{thm:nbsGS}, one can prove without much difficulty that the optimal solution of the following minimization subproblem
\begin{eqnarray}
     \min_{\x\in\cX}\; \Big\{p(x_1)+q(\x)
  +\frac{1}{2}\norm{\x  - \bar \x }_{\cT_{\rm sSOR}}^2 -\inprod{\x}{\Delta_{\rm sSOR}(\mbox{\boldmath{$\delta'$}},\mbox{\boldmath{$\delta$}})} \Big\},
 \label{prox-sSOR}
\end{eqnarray}
can be computed by performing exactly
one cycle of the block sSOR method. In particular, when $p(\cdot) \equiv 0$ and $\bdelta = \bdelta' = \mb0$, the optimal solution to \eqref{prox-sSOR} can be computed by \eqref{eq-SSOR}, i.e., set $\bar \x = \x^k$, then $\x^{k+1}$ obtained from
\eqref{eq-SSOR} is the optimal solution to \eqref{prox-sSOR}.
By replacing $\cT_\cQ$ and $\Delta(\cdot,\cdot)$ in Algorithm 1 with $\cT_{\rm sSOR}$ and $\Delta_{\rm sSOR}(\cdot, \cdot)$, respectively, one can obtain a block sSOR based inexact proximal gradient method for solving
\eqref{eq-QP} and  the convergence results presented in
Proposition \ref{prop:convergence_sgs}   and Theorem \ref{thm:convergence_rate} still remain valid with
$\widehat{\cQ}$ replaced by $\widehat{\cQ}_\omega$.

\begin{remark} For the classical { pointwise} sSOR method, it was shown in
\cite[Theorem 4.8.14]{Hackbusch} that  if there exist positive constants $\gamma$ and $\Gamma$ such that
\begin{eqnarray*}
  0 \prec \gamma \cD \preceq \cQ, \quad \Big(\frac{1}{2}\cD+\cU\Big)\cD^{-1}\Big(\frac{1}{2}\cD + \cU^*\Big) \preceq
  \frac{\Gamma}{4} \cQ,
\end{eqnarray*}
then its convergence rate is
$\norm{I - \cQ^{1/2} \widehat{\cQ}_\omega^{-1}\cQ^{1/2}}_2 \leq 1 - \frac{2\bar{\tau}}{\bar{\tau}^2/\gamma + \bar{\tau}+\Gamma/4},$ where $\bar{\tau} = 1/\omega  - 1/2.$ Interestingly, for the convergence rate of
 our block sSOR method in
 Theorem \ref{thm:convergence_rate}, we also have a similar estimate given by
 \begin{eqnarray*}
  \norm{I -  \widehat{\cQ}_\omega^{-1/2}\cQ \widehat{\cQ}_\omega^{-1/2}}_2
  \leq 1 - \frac{2\bar{\tau}}{\bar{\tau}^2/\gamma + \bar{\tau}+\Gamma/4}.
 \end{eqnarray*}
 In order to minimize the upper bound, we can choose
 $\omega_* = 2/(1+\sqrt{\gamma \Gamma})$ and then we get
 \begin{eqnarray*}
  \norm{I -  \widehat{\cQ}_{\omega_*}^{-1/2}\cQ \widehat{\cQ}_{\omega_*}^{-1/2}}_2
 \; \leq\; \frac{1-\sqrt{\gamma/\Gamma}}{1+\sqrt{\gamma/\Gamma}}.
 \end{eqnarray*}
\end{remark}

\section{Conclusion}

 {In this paper, we
 give an optimization
interpretation that each cycle of the classical block sGS method is equivalent to solving the associated
multi-block convex QP problem with an additional proximal term. This equivalence is fully characterized via our  block sGS decomposition theorem. A factorization view of this theorem and its equivalence to the SCB reduction procedure are also established.
The classical block sGS method, viewed from the optimization perspective  via the block sGS decomposition theorem,
is then extended to the inexact setting for solving a class of
multi-block
convex composite QP problems involving nonsmooth functions. Moreover, we are able to derive $O(1/k)$ and $O(1/k^2)$ iteration complexities for our inexact block sGS method and its accelerated version, respectively. These new interpretations and convergence results, together with the incorporation of the (inexact) sGS decomposition techniques in the design of efficient algorithms for core optimization problems  in \cite{CST17,DWD,QSDPNAL,SCBADMM,STY16}, demonstrate the power and usefulness of our simple yet elegant block sGS decomposition theorem.
We believe this decomposition theorem will be proven to be even more
{useful} 
 in solving {other} optimization problems and beyond.

\section*{Appendix: Proof of part (b) of Proposition \ref{prop:convergence_sgs}}

To begin the proof, we state the following lemma from \cite{SRB}.
\begin{lemma} \label{lem-ineq}
 Suppose that $\{u_k\}$ and $\{\lam_k\}$  are two sequences of  nonnegative scalars,
 and $\{ s_k\}$ is a nondecreasing sequence of scalars such that $s_0\geq u_0^2$.
  Suppose that  for all $k\geq 1$, the  inequality
 $
  u_k^2 \leq s_k + 2\sum_{i=1}^k \lam_i u_i
 $ holds.
 Then for all $k\geq 1$,
 $
  u_k \leq \bar{\lam}_k + \sqrt{ s_k + \bar{\lam}_k^2},
$
 where $\bar{\lam}_k = \sum_{i=1}^k \lam_i$.
\end{lemma}

\begin{proof}
In this proof, we let $\Delta^j = \Delta(\tilde{\bdelta}^j,\bdelta^j)$. Note that under the assumption
that $t_j=1$ for all $j\geq 1$, $\widetilde{\x}^j = \x^{j-1}$.
Note also that from \eqref{eq-error}, we have that $\norm{\hcQ^{-1/2}\Delta^j} \leq M \eps_j$,
where $M$ is given as in Proposition \ref{prop:convergence_sgs}.

From the optimality of $\x^j$ in \eqref{eq-apg}, one can show that
\begin{eqnarray}
 F(\x) - F(\x^j) \geq \frac{1}{2}\norm{\x^j - \x^{j-1}}_{\hcQ}^2 + \inprod{\x^{j-1}-\x}{\hcQ(\x^j-\x^{j-1})} + \inprod{\Delta^j}{\x-\x^j} \quad \forall\;\x.
 \label{eq-A-1}
\end{eqnarray}
Let $\e^j = \x^j-\x^*$.
By setting $\x= \x^{j-1}$ and $\x=\x^*$ in \eqref{eq-A-1}, we get
\begin{eqnarray}
   F(\x^{j-1}) - F(\x^j) & \geq & \frac{1}{2} \norm{\e^j-\e^{j-1}}_\hcQ^2 + \inprod{\Delta^j}{\e^{j-1}-\e^j},
\label{eq-A-2a}   \\[5pt]
   F(\x^*) - F(\x^j) &\geq & \frac{1}{2} \norm{\e^j}_\hcQ^2 -\frac{1}{2} \norm{\e^{j-1}}_\hcQ^2 - \inprod{\Delta^j}{\e^j}.
   \label{eq-A-2b}
\end{eqnarray}
By multiplying $j-1$ to \eqref{eq-A-2a} and combining with \eqref{eq-A-2b}, we get
\begin{eqnarray}
 (a_j + b_j^2)  &\leq&  (a_{j-1}+b_{j-1}^2) - (j-1)\norm{\e^j-\e^{j-1}}_\hcQ^2
 + 2\inprod{\Delta^j}{ j \e^j-(j-1)\e^{j-1} }
\nn \\[5pt]
 &\leq & (a_{j-1}+b_{j-1}^2)  + 2 \norm{\hcQ^{-1/2}\Delta^j}\norm{ j \e^j-(j-1)\e^{j-1} }_\hcQ
\nn \\[5pt]
 &\leq & (a_{j-1}+b_{j-1}^2)  +  2\norm{\hcQ^{-1/2}\Delta^j}(j b_j + (j-1)b_{j-1})
 \nn \\
 &\leq & \cdots
\nn  \\
 &\leq & a_1 + b_1^2 + 2 \sum_{i=2}^j M\eps_i  (i b_i +(i-1)b_{i-1})
 \;\leq \;
  b_0^2 + 2 \sum_{i=1}^j 2Mi \eps_i   b_i  ,\label{eq-A-3}
\end{eqnarray}
where $a_j = 2j [F(\x^j)-F(\x^*)]$ and $b_j = \norm{\e^j}_\hcQ$. Note that the last inequality follows
from \eqref{eq-A-2b} with $j=1$ and some simple manipulations.
To summarize, we have $b_j^2  \leq  b_0^2 + 2 \sum_{i=1}^j 2M i\eps_i   b_i $. By applying
Lemma \ref{lem-ineq}, we get
\begin{eqnarray*}
  b_j  \;\leq\;   \bar{\lam}_j + \sqrt{ b_0^2 + \bar{\lam}_j^2} \;\leq\;  b_0 + 2\bar{\lam}_j,
\end{eqnarray*}
where $\bar{\lam}_j =  \sum_{i=1}^j \lam_i$ with $\lam_i = 2M i \eps_i$.  Applying the above result to
\eqref{eq-A-3}, we get
\begin{eqnarray*}
  a_j \;\leq\;  b_0^2 + 2 \sum_{i=1}^j \lam_i (2\bar{\lam}_i + b_0) \; \leq\;  (b_0 + 2\bar{\lam}_j)^2.
\end{eqnarray*}
From here, the required result in Part (b) of Proposition \ref{prop:convergence_sgs} follows.
\end{proof}


\end{document}